\documentclass[11pt,a4paper]{article}
\usepackage{amsmath}
\usepackage{amssymb}
\usepackage{amsfonts}
\usepackage{amsthm}
\usepackage{relsize}
\usepackage{setspace}
\usepackage{geometry}
\usepackage{url}
\usepackage{enumerate}

\input xy
\xyoption{all}

\begin{document}

\newcounter{saveenum}

\newtheorem{thm}{Theorem}[section]
\newtheorem{lem}[thm]{Lemma}
\newtheorem{prop}[thm]{Proposition}
\newtheorem{cor}[thm]{Corollary}
\newtheorem{hyp}[thm]{Hypothesis}
\newtheorem*{propn}{Proposition}
\newtheorem*{corn}{Corollary}

\theoremstyle{definition}
\newtheorem{defn}[thm]{Definition}
\newtheorem{eg}[thm]{Example}
\newtheorem*{egn}{Example}
\newtheorem{que}{Question}
\newtheorem{prob}[que]{Problem}

\theoremstyle{remark}
\newtheorem{rem}{Remark}

\newcommand{\Hpi}{\mathrm{Hall}_\pi}
\newcommand{\Hpip}{\mathrm{Hall}_{\pi'}}
\newcommand{\Irr}{\mathrm{Irr}}

\newcommand{\sgp}{{[\leq]}}
\newcommand{\nsgp}{{[\unlhd]}}
\newcommand{\fsgp}{{[\leq_f]}}
\newcommand{\nil}{{[\mathrm{pronil}]}}
\newcommand{\sol}{{[\mathrm{prosol}]}}
\newcommand{\fin}{{[\mathrm{fin}]}}
\newcommand{\fsg}{{[\mathrm{sim}]}}
\newcommand{\elcyc}{{[\mathrm{elcyc}]}}

\newcommand{\ir}{\varrho}

\newcommand{\und}{\underbar}
\newcommand{\Sylp}{\mathrm{Syl}_p}
\newcommand{\piP}{p \in \mathbb{P}}
\newcommand{\la}{\leftarrow}
\newcommand{\IH}{\mathrm{IH}}
\newcommand{\Cf}{\mathrm{Cf}}
\newcommand{\Aut}{\mathrm{Aut}}
\newcommand{\Inn}{\mathrm{Inn}}
\newcommand{\Out}{\mathrm{Out}}
\newcommand{\Comm}{\mathrm{Comm}}
\newcommand{\KComm}{\mathrm{KComm}}
\newcommand{\LComm}{\mathrm{LComm}}
\newcommand{\ICom}{\mathrm{ICom}}
\newcommand{\SNp}{\mathcal{S}}
\newcommand{\e}{\mathrm{e}}
\newcommand{\LFe}{\mathrm{LF-e}}
\newcommand{\db}{\mathrm{db}}
\newcommand{\eb}{\mathrm{eb}}
\newcommand{\Sym}{\mathrm{Sym}}
\newcommand{\Alt}{\mathrm{Alt}}
\newcommand{\ch}{\mathrm{char}}
\newcommand{\PA}{\mathrm{PA}}
\newcommand{\rFin}{\mathrm{Fin}}
\newcommand{\Hom}{\mathrm{Hom}}
\newcommand{\imm}{\mathrm{imm}}
\newcommand{\Core}{\mathrm{Core}}
\newcommand{\rQ}{\mathrm{Q}}
\newcommand{\id}{\mathrm{id}}
\newcommand{\lep}{\leq_{[p]}}
\newcommand{\plh}{\mathrm{ht}_{[p]}}
\newcommand{\Cl}{\mathrm{Cl}}
\newcommand{\GL}{\mathrm{GL}}
\newcommand{\SL}{\mathrm{SL}}
\newcommand{\PSL}{\mathrm{PSL}}
\newcommand{\ord}{\mathrm{ord}^\times}
\newcommand{\pd}{\mathrm{pd}}
\newcommand{\Comp}{\mathrm{Comp}}
\newcommand{\St}{\mathrm{St}}
\newcommand{\Emb}{\mathcal{E}_{p'}}
\newcommand{\EmbLF}{\mathcal{E}^{\mathrm{LF}}_{p'}}
\newcommand{\EmbS}{\mathcal{E}^{\mathrm{sep}}_{p'}}
\newcommand{\CEmbL}{\mathcal{C}^{\mathrm{L}}_{p'}}
\newcommand{\CEmbLF}{\mathcal{C}^{\mathrm{LF}}_{p'}}
\newcommand{\CEmbA}{\mathcal{C}^{\mathrm{ab}}_{p'}}
\newcommand{\CEmbC}{\mathcal{C}^{\mathrm{crit}}_{p'}}
\newcommand{\CEmbS}{\mathcal{C}^{\mathrm{sep}}_{p'}}
\newcommand{\wrd}{\mathrm{wrd}}
\newcommand{\Qd}{\mathrm{Qd}}
\newcommand{\Sp}{\mathrm{Sp}}
\newcommand{\tp}{\mathrm{top}}

\newcommand{\Ob}{\mathrm{Ob}}
\newcommand{\ob}{\mathrm{ob}}
\newcommand{\OI}{\mathrm{OI}}
\newcommand{\Ilhd}{\mathrm{I}^\lhd}

\newcommand{\rC}{\mathrm{C}}
\newcommand{\rE}{\mathrm{E}}
\newcommand{\rF}{\mathrm{F}}
\newcommand{\rH}{\mathrm{H}}
\newcommand{\rN}{\mathrm{N}}
\newcommand{\rO}{\mathrm{O}}
\newcommand{\rZ}{\mathrm{Z}}

\newcommand{\mcA}{\mathcal{A}}
\newcommand{\mcB}{\mathcal{B}}
\newcommand{\mcC}{\mathcal{C}}
\newcommand{\mcD}{\mathcal{D}}
\newcommand{\mcE}{\mathcal{E}}
\newcommand{\mcF}{\mathcal{F}}
\newcommand{\mcG}{\mathcal{G}}
\newcommand{\mcH}{\mathcal{H}}
\newcommand{\mcI}{\mathcal{I}}
\newcommand{\mcJ}{\mathcal{J}}
\newcommand{\mcK}{\mathcal{K}}
\newcommand{\mcL}{\mathcal{L}}
\newcommand{\mcM}{\mathcal{M}}
\newcommand{\mcN}{\mathcal{N}}
\newcommand{\mcO}{\mathcal{O}}
\newcommand{\mcP}{\mathcal{P}}
\newcommand{\mcQ}{\mathcal{Q}}
\newcommand{\mcR}{\mathcal{R}}
\newcommand{\mcS}{\mathcal{S}}
\newcommand{\mcT}{\mathcal{T}}
\newcommand{\mcU}{\mathcal{U}}
\newcommand{\mcV}{\mathcal{V}}
\newcommand{\mcW}{\mathcal{W}}
\newcommand{\mcX}{\mathcal{X}}
\newcommand{\mcY}{\mathcal{Y}}
\newcommand{\mcZ}{\mathcal{Z}}

\newcommand{\bC}{\mathbb{C}}
\newcommand{\bF}{\mathbb{F}}
\newcommand{\bN}{\mathbb{N}}
\newcommand{\bP}{\mathbb{P}}
\newcommand{\bQ}{\mathbb{Q}}
\newcommand{\bR}{\mathbb{R}}
\newcommand{\bZ}{\mathbb{Z}}

\newcommand{\bigast}{\ensuremath{\displaystyle\mathop{\mathlarger{\ast}}}}

\setlength{\parindent}{0pt}

\setlength{\parskip}{3mm}

\title{The number of profinite groups with a specified Sylow subgroup}

\author{Colin D. Reid\\
University of Newcastle\\
School of Mathematical and Physical Sciences\\
University Drive, Callaghan NSW 2308\\
Australia\\
colin@reidit.net}

\maketitle

\begin{abstract}Let $S$ be a finitely generated pro-$p$ group.  Let $\Emb(S)$ be the class of profinite groups $G$ that have $S$ as a Sylow subgroup, and such that $S$ intersects non-trivially with every non-trivial normal subgroup of $G$. In this paper, we investigate the question of whether or not $\Emb(S)$ has finitely many isomorphism classes.  For instance, we give an example where $\Emb(S)$ contains an infinite ascending chain of soluble groups, and on the other hand show that $\Emb(S)$ contains only finitely many isomorphism classes in the case that $S$ is just infinite.\end{abstract}


\section{Introduction}

Groups of prime power order are a pervasive feature of finite group theory.  This is clearest in Sylow's theorem and more generally in the theory of fusion (also known as local analysis).  The immediate goal is to understand the manner in which a $p$-group can be embedded in a finite group, especially with regard to the normalisers of its subgroups, as a tool for understanding finite groups by means of the $p$-groups contained in them.  The theory of fusion in finite groups is well-developed, and in particular played a large role in the classification of finite simple groups.  It has also developed into a more general theory of fusion systems of finite $p$-groups, which do not necessarily arise from fusion within a finite group.  (See \cite{Cra} for an account of this theory.)

Sylow's theorem generalises directly to profinite groups: in a profinite group $G$, every pro-$p$ subgroup is contained in a maximal pro-$p$ subgroup, which we call a $p$-Sylow subgroup, all $p$-Sylow subgroups are conjugate, and if $S$ is a $p$-Sylow subgroup of $G$ then $SN/N$ is a $p$-Sylow subgroup of $G/N$ for every (finite or profinite) quotient of $G$.  In principle, the theory of fusion can be developed for profinite groups in much the same way as for finite groups.  Indeed, the fact that pro-$p$ groups are generally better understood than profinite groups would suggest this as an approach for extending results from the former class to the latter.  However, fusion theory is much less developed for profinite groups than for finite groups.  As far as the author is aware, the first significant foray into this area was a paper by Gilotti, Ribes and Serena (\cite{GRS}); since then, fusion and fusion systems in a profinite context have also been developed by Stancu and Symonds (see \cite{SS} and \cite{Sym}).

A basic problem in this area is to understand the profinite groups that have a given $p$-Sylow subgroup $S$.  Write $p'$ for the set of primes other than $p$.  Any profinite group $G$ has a unique largest normal pro-$p'$ subgroup $\rO_{p'}(G)$, the \emph{$p'$-core} of $G$.  From the point of view of the associated fusion system on $S$ (that is, the category of homomorphisms between closed subgroups of $S$ that are induced by conjugation in $G$), the $p'$-core plays no role, in that fusion in a $p$-Sylow subgroup of $G$ is equivalent to fusion in a $p$-Sylow subgroup of $G/\rO_{p'}(G)$.  In any case, the $p$-Sylow subgroups of $G$ impose no meaningful restriction on the structure of $\rO_{p'}(G)$, for instance we could have $G = S \times H$ where $H$ is any pro-$p'$ group.  So we are left with the following problem.

\begin{prob}Let $S$ be a pro-$p$ group.  Let $\Emb(S)$ be the class of profinite groups that have $S$ as a $p$-Sylow subgroup and have no non-trivial normal pro-$p'$ subgroups.  Describe $\Emb(S)$ in terms of internal properties of $S$.\end{prob}

A natural question to ask here is the following:

\begin{que}\label{mainque}For which pro-$p$ groups $S$ does $\Emb(S)$ contain infinitely many isomorphism classes of profinite group?\end{que}

This question, and variants of it, will be the focus of this paper.  For the purposes of this paper, all subgroups are required to be closed and all homomorphisms are required to be continuous, and a `finite' class of groups is one that contains finitely many isomorphism classes of topological groups.  We will concentrate on the case that $S$ is (topologically) finitely generated, which appears to be more tractable.  The following can be deduced from a theorem of Tate:

\begin{lem}\label{tatefin}Let $S$ be a finitely generated pro-$p$ group.  Then every group in $\Emb(S)$ is virtually pro-$p$.\end{lem}

If $G \in \Emb(S)$, then there is a subgroup $P$ of $S$ which is open and normal in $G$; now $P$ is also finitely generated, so $\Phi(P)$ is also open in $G$.  It follows from some basic extension theory that $G$ is determined as an element of $\Emb(S)$ by the quotient $G/P$ together with its action on $P/\Phi(P)$:

\begin{thm}\label{extnthm}Let $P$ be a finitely generated pro-$p$ group, and let $K$ be a finite group.  Suppose the extensions
\[\xymatrix{1 \ar[r] & P \ar[r] & G \ar[r] & K \ar[r] & 1 }\]
and
\[\xymatrix{1 \ar[r] & P \ar[r] & G^* \ar[r] & K \ar[r] & 1 }\]
admit a common restriction
\[\xymatrix{1 \ar[r] & P \ar[r] & S \ar[r] & T \ar[r] & 1 }\]

where $T$ is a $p$-Sylow subgroup of $K$, and the action of $K$ on $P/\Phi(P)$ is the same in both extensions.

Then the extensions are equivalent, and hence $G \cong G^*$.\end{thm}

\begin{cor}\label{embsizecor}Let $S$ be a finitely generated pro-$p$ group.  Then for all $n$, the number of isomorphism types of profinite group $G$ having $S$ as a Sylow subgroup of index at most $n$ is finite.  In particular, $\Emb(S)$ is at most countably infinite, and $\Emb(S)$ is finite if and only if there is an overall bound on $|G:S|$ for all $G \in \Emb(S)$.\end{cor}

So Question~\ref{mainque} is equivalent to asking whether there is a bound on $|G:S|$ (or equivalently on $|G:\rO_p(G)|$, or on $|G:\Phi(\rO_p(G))|$).

It is also of interest to consider two more restricted classes of $p'$-embeddings:

\begin{defn}Let $G$ be a profinite group.  A \emph{component} of $G$ is a subnormal subgroup $Q$ such that $Q$ is perfect and $Q/\rZ(Q)$ is simple.  (Note that these conditions ensure that $Q$ is finite.)  Define the \emph{layer} $\rE(G)$ of $G$ to be the closed subgroup of $G$ generated by the components of $G$.  Given a pro-$p$ group $S$, define $\EmbLF(S)$ to be the class of groups $G \in \Emb(S)$ such that $\rE(G)=1$.  Define $\EmbS(S)$ to be the class of groups in $\Emb(S)$ that are $p$-separable, that is, which have no non-abelian composition factors of order divisible by $p$.

The \emph{pro-Fitting subgroup} $\rF(G)$ of $G$ is the unique largest normal pronilpotent subgroup of $G$.  The \emph{generalised pro-Fitting subgroup} $ \rF^*(G)$ of $G$ is given by $ \rF^*(G) = \rF(G)\rE(G)$.\end{defn}

In a virtually pronilpotent group, the generalised pro-Fitting subgroup contains its own centraliser (see \cite{ReiF}), so if $G \in \EmbLF(S)$ for a finitely generated pro-$p$ group $S$, then $\rO_p(G)$ contains its own centraliser in $G$, and indeed $G/\rO_p(G)$ acts faithfully on $\rO_p(G)/\Phi(\rO_p(G))$.  So if $S$ is finite, or more generally if $S$ has finite subgroup rank, then we obtain a bound on $|G/\rO_p(G)|$, so $\EmbLF(S)$ is finite.  Even in this case it can happen that $\Emb(S)$ is infinite: for instance, $S$ may be the $p$-Sylow subgroup of infinitely many finite simple groups.  More interesting is the case when $\EmbLF(S)$ or $\EmbS(S)$ is infinite.  Consider for instance the following:

\begin{prop}\label{ascembed}Let $p$ and $q$ be distinct primes.  Then there is a there is a $2$-generator metabelian pro-$p$ group $S$ and an infinite ascending chain
\[ S < G_0 < G_1 < G_2 < \dots\]
of profinite groups, each open in the next, with the following properties:

the union $G = \bigcup_{i \ge 0} G_i$ is a soluble group of derived length $3$, and $G=SQ$ where $Q$ is a countably infinite discrete elementary abelian $q$-group;

for all $i \ge 0$, $\rO_{p'}(G_i)=1$, so $G_i \in \EmbS(S)$;

the fusion systems $\mcF_{G_i}(S)$ are pairwise non-isomorphic; indeed, the fusion of conjugacy classes of $S$ in $G_i$ and $G_j$ is inequivalent for all $i \not= j$.\end{prop}

Nevertheless, there are significant restrictions on the structure of $p'$-embeddings of $2$-generator pro-$p$ groups (See Theorem~\ref{2genthm} below).  The reason for this is the role played normal subgroups $P$ of a pro-$p$ group $S$ that are not contained in $\Phi(S)$, and in the $2$-generator case, $P \not\le \Phi(S)$ implies $S/P$ is cyclic (in particular, $P \ge S'$).  Indeed, for groups $S$ such that $P \not\le \Phi(S)$ for only finitely many normal subgroups $P$, we obtain the following:

\begin{thm}\label{obphithm}Let $S$ be an infinite finitely generated pro-$p$ group.  Let $\mcK$ be the set of open normal subgroups of $S$ that are not contained in $\Phi(S)$.  Suppose that $\mcK$ is finite.  Then $\Emb(S) = \EmbLF(S)$ and $\EmbS(S)$ is finite.  If in addition $|S:S^{(n)}|$ is finite for all $n$, then $\Emb(S)$ is finite.\end{thm}

The hypotheses of Theorem~\ref{obphithm} are immediately satisfied if $S$ is generated by $2$ elements and $|S:S^{(n)}|$ is finite for all $n$, because the order of a cyclic quotient is at most $|S:S'|$.  The hypotheses of Theorem \ref{obphithm} are also satisfied by all just infinite pro-$p$ groups of infinite subgroup rank.  As a result we obtain the following:

\begin{thm}\label{jiemb}Let $S$ be a just infinite pro-$p$ group.  Then $\Emb(S)$ is finite.  In other words, only finitely many just infinite groups have $S$ as a Sylow subgroup.\end{thm}

In general, for a given finitely generated pro-$p$ group $S$, the question of whether $\Emb(S)$, $\EmbLF(S)$ or $\EmbS(S)$ is finite reduces to considering $p'$-embeddings of more restricted types (see Theorem~\ref{critthm}).  We also obtain several restrictions (Theorem~\ref{wreg}) on the structure of groups in $\Emb(S)$ in the case that $S$ is weakly regular, that is, $S$ does not have a quotient isomorphic to $C_p \wr C_p$.  This class of pro-$p$ groups includes for instance all nilpotent pro-$p$ groups of class less than $p$ and all powerful pro-$p$ groups.  It is not known if there are any finitely generated weakly regular pro-$p$ groups $S$ for which $\EmbLF(S)$ is infinite.

\section{Preliminaries}

We gather here some basic facts and definitions we will need about finite and profinite groups.

\begin{defn}Let $G$ be a profinite group.  Define $d(G)$ to be the size of the smallest subset $X$ of $G$ such that $G = \overline{\langle X \rangle}$.  Say $G$ is \emph{$n$-generated} if $d(G) \le n$.

Define $G'$ to be the closed commutator subgroup $\overline{[G,G]}$, and define $G^{(n)}$ inductively by $G^{(0)} = G$ and $G^{(n+1)} = (G^{(n)})'$.  Write $G^n$ for the smallest closed subgroup of $G$ containing all $n$-th powers in $G$.

Given a prime (or set of primes) $p$, the \emph{$p$-core} $\rO_p(G)$ is the largest normal pro-$p$ subgroup of $G$, and the \emph{$p$-residual} $\rO^p(G)$ is the smallest normal subgroup of $G$ such that $G/\rO^p(G)$ is a pro-$p$ group.\end{defn}

\begin{lem}\label{goodcomp}Let $G$ be a profinite group and let $\mcQ$ be a set of components of $G$.  Then $K = \overline{\langle \mcQ \rangle}$ is a central product of $\mcQ$ and no proper subset of $\mcQ$ suffices to generate $K$ topologically.  Every component of $G$ is contained in a finite normal subgroup of $G$.\end{lem}

\begin{proof}See \cite{ReiF} Proposition 2.8.\end{proof}

\begin{lem}\label{cinvlem}Let $P$ be a finitely generated pro-$p$ group and let $G = P \rtimes H$ be a profinite group such that $\rC_H(P)=1$.
\begin{enumerate}[(i)]
\item Suppose there is an $H$-invariant series
\[ P = P_1 \ge P_2 \ge \dots \]
of normal subgroups of $P$, such that $\bigcap P_i = 1$, and such that $[P_i,H] \le P_{i+1}$ for each $i$.  Then $H$ is a pro-$p$ group.
\item Define the characteristic series $P_i$ by $P_1 = P$, and thereafter $P_{i+1} = [P,P_i]P^p_i$.  Suppose $H$ acts trivially on $P/\Phi(P)$.  Then $H$ acts trivially on $P_i/P_{i+1}$ for all $i$.  In particular, $H$ is a pro-$p$ group.
\item Suppose $P$ is finite and abelian, and $H$ is a $p'$-group.  Then $P = [P,H] \times \rC_P(H)$.
\end{enumerate}
\end{lem}

\begin{proof}For parts (i) and (ii) see \cite{Lee} Exercise 2.1 (2); the generalisation to profinite groups is immediate.  For part (iii) see \cite{Asc} Proposition 24.6.\end{proof}

\begin{lem}\label{gfitlem}Let $G$ be a profinite group that is virtually pronilpotent.  Then ${\rC_G( \rF^*(G)) = \rZ(\rF(G))}$.\end{lem}

\begin{proof}This is a special case of \cite{ReiF}, Theorem 1.7.\end{proof}

\begin{cor}\label{lflem}Let $S$ be a finitely generated pro-$p$ group, let $G \in \EmbLF(S)$ and let ${P = \rO_p(G)}$.  Then $G/P$ acts faithfully on $P/\Phi(P)$.  As a result, we have $H \in \EmbLF(S)$ for all closed subgroups $H$ of $G$ containing $S$.\end{cor}

\begin{proof}By Lemma~\ref{gfitlem}, we have $\rC_G(P) \le P$.  By Lemma~\ref{cinvlem}, the section ${\rC_G(P/\Phi(P))/\rC_G(P)}$ is a pro-$p$ group, so $\rC_G(P/\Phi(P))$ is a pro-$p$ group; since $\Phi(P) \ge P'$, we have $P \le C_G(P/\Phi(P))$.  But $P$ is the largest normal pro-$p$ subgroup of $G$, so in fact $P = \rC_G(P/\Phi(P))$.

Now let $H$ be a subgroup of $G$ containing $S$.  Clearly $S$ is a $p$-Sylow subgroup of $H$.  We have $\rC_H(\rO_p(H)) \le \rC_H(P) \le P$, since $P$ is a normal pro-$p$ subgroup of $H$.  This ensures that $\rE(H)$ and $\rO_{p'}(H)$ are both trivial.  Thus $H \in \EmbLF(S)$.\end{proof}

\begin{defn}Let $P$ be a finite $p$-group.  A characteristic subgroup $K$ of $P$ is \emph{critical} if $[P,K]\Phi(K) \le \rZ(K)$ and $\rC_P(K) = \rZ(K)$.\end{defn}

\begin{thm}[Thompson, \cite{Fei} Chapter II, Lemma 8.2]\label{thomcrit}Let $P$ be a finite $p$-group.  Then $P$ has a critical subgroup.  If $K$ is a critical subgroup of $P$, then the kernel of the induced homomorphism $\Aut(P) \rightarrow \Aut(K)$ is a $p$-group.\end{thm}

\section{Control of $p$-transfer in profinite groups}\label{ptransprelim}

An important notion in finite group theory is the \emph{transfer map}, which is a homomorphism that is defined from a finite group to any of its abelian sections.   We will not be using the transfer map directly, but we will be using the closely related notion of control of transfer, and more precisely control of $p$-transfer.  Control of transfer is a concept that behaves well in the class of profinite groups; see for instance \cite{GRS}.  (Note however that our definition of which subgroup controls transfer is slightly different to that used in \cite{GRS}.)

\begin{defn}Let $G$ be a profinite group, let $H$ be a subgroup, and let $H \leq K \leq G$.  Say $K$ \emph{controls transfer} from $G$ to $H$ if $G' \cap H = K' \cap H$.  In the special case that $H$ is a $p$-Sylow subgroup of $G$, then say $K$ \emph{controls $p$-transfer} in $G$.  There is a potential ambiguity in saying that $K$ controls $p$-transfer in $G$ without specifying the Sylow subgroup, but since all Sylow subgroups of $G$ contained in $K$ are conjugate in $K$, the choice of Sylow subgroup is immaterial in practice.\end{defn}

The theorem below is an interpretation essentially due to Gagola and Isaacs (\cite{Gag}) of a theorem of Tate (\cite{Tat}).  Both \cite{Tat} and \cite{Gag} state the result for finite groups, but the generalisation to profinite groups is immediate.

\begin{thm}[Tate]\label{tate}Let $G$ be a (pro-)finite group, let $S$ be a $p$-Sylow subgroup of $G$, and let $S \leq K \leq G$.  The following are equivalent:
\begin{enumerate}[(i)]
\item $G' \cap S = K' \cap S$;
\item $(G'G^p) \cap S = (K'K^p) \cap S$;
\item $(G'\rO^p(G)) \cap S = (K'\rO^p(K)) \cap S$;
\item $\rO^p(G) \cap S = \rO^p(K) \cap S$.
\end{enumerate}
\end{thm}

From now on, the statement `$K$ controls $p$-transfer in $G$' will be taken to mean any of the four equations above interchangeably.

In a profinite group $G$, a \emph{normal $p$-complement} is a (necessarily unique) normal subgroup $N$ such that $G = SN$ and $S \cap N = 1$, where $S$ is a $p$-Sylow subgroup of $G$.  Theorem \ref{tate} has some immediate consequences for normal $p$-complements in normal subgroups of (pro-)finite groups (indeed, this was the original motivation of Tate's result in the finite context).

\begin{cor}\label{tatecor}Let $G$ be a profinite group, and let $S \in \Sylp(G)$.
\begin{enumerate}[(i)]
\item Let $M$ be a normal subgroup of $G$ such that $S \cap M \leq \Phi(S)$.  Then $SM$ has a normal $p$-complement, and $\rO_{p'}(G/M) = \rO_{p'}(G)M/M$.
\item Let $M$ and $N$ be normal subgroups of $G$ such that $S \cap M \leq \Phi(S)N$.  Then $MN/N$ has a normal $p$-complement.
\end{enumerate}
\end{cor}

\begin{proof}(i) For any normal subgroup $M$ of $G$, we have $(SM)'(SM)^p = \Phi(S)M$.  The condition $S \cap M \le \Phi(S)$ then implies
\[ ((SM)'(SM)^p) \cap S = \Phi(S)M \cap S = \Phi(S) = S'S^p.\]
Hence by Theorem \ref{tate} we have $\rO^p(SM) \cap S = \rO^p(S) \cap S = 1$, in other words $\rO^p(SM)$ is the normal $p$-complement of $SM$.  Note that $\rO^p(SM)$ is also a normal $p$-complement in $M$.

For the final assertion, let $O$ be the lift of $\rO_{p'}(G/M)$ to $G$.  It is clear that $O \ge \rO_{p'}(G)M$.  On the other hand, $S \cap O = S \cap M \le \Phi(S)$, so $O$ has a normal $p$-complement $K$, by the same argument as for $M$.  Since $M$ contains a $p$-Sylow subgroup of $O$, we have $O=KM$; since $K$ is a normal pro-$p'$ subgroup of $G$, we have $K \le O_{p'}(G)$, so $O=\rO_{p'}(G)M$.

(ii) $MN/N$ is a normal subgroup of $G/N$, and $\Phi(S/N) = \Phi(S)N/N$ contains ${(M \cap S)N/N}$.  The result follows by part (i) applied to $G/N$.\end{proof}

\begin{proof}[Proof of Lemma \ref{tatefin}]Since $\Phi(S)$ is open in $S$, there is some open normal subgroup $N$ of $G$ such that $S \cap N \leq \Phi(S)$.  By Corollary \ref{tatecor}, $N/\rO_{p'}(N)$ is a pro-$p$ group, so $G/\rO_{p'}(N)$ is virtually pro-$p$.  Now $\rO_{p'}(N) \leq \rO_{p'}(G)$, so $G/\rO_{p'}(G)$ is an image of $G/\rO_{p'}(N)$; hence $G/\rO_{p'}(G)$ is virtually pro-$p$.\end{proof}

It is worth noting in particular a sufficient condition under which every $p'$-embedding is layer-free.

\begin{cor}\label{finlfcor}Let $S$ be a finitely generated pro-$p$ group and let $G \in \Emb(S)$.  Suppose that $\Phi(S)$ contains every finite normal subgroup of $S$.  Then $\rE(G)=1$.\end{cor}

\begin{proof}Certainly $\rE(G)$ is finite, since $G$ is virtually pro-$p$ by Lemma~\ref{tatefin}, so $\rE(G) \cap S$ is a finite normal subgroup of $S$.  Additionally, $p$ divides the order of every component of $G$, since $\rO_{p'}(G)=1$.  But $\rE(G) \cap S \le \Phi(S)$, so $\rE(G)$ has a normal $p$-complement.  Hence $\rE(G)=1$.\end{proof}

\begin{defn}Let $S$ be a finitely-generated pro-$p$ group and let $G$ be a $p'$-embedding of $S$.  Say $G$ is \emph{Frattini} if $\rO_p(G) \le \Phi(S)$, or more generally, say $G$ is \emph{quasi-Frattini} if $\rO_p(G) \cap \Phi(S)$ is normal in $G$.  Say $G$ is \emph{standard} if $\rO_p(G) \cap \Phi(S)$ is not normal in $G$.

Given a profinite group $G$, define the \emph{$p$-layer} $\rE_p(G)$ to be the set of components of $G$ of order divisible by $p$.  (Note that if a quasisimple group $Q$ is of order divisible by $p$, then the simple quotient $Q/\rZ(Q)$ is also of order divisible by $p$.)\end{defn}

\begin{lem}\label{centauto}Let $G$ be a (topological) group and let $\alpha$ be an automorphism of $G$ (as a topological group) that acts trivially on $G/\rZ(G)$.  Then $\alpha$ acts trivially on $G'$.  In particular, if $G$ is (topologically) perfect then $\Aut(G)$ acts faithfully on $G/\rZ(G)$.\end{lem}

\begin{proof}Let $\alpha$ be an automorphism of $G$ and write $[\alpha,x]$ for $x\alpha(x^{-1})$.  Suppose $[\alpha,x] \in \rZ(G)$ for all $x \in G$.  Then for all $x,y \in G$, we have the following:
\[ \alpha([x,y]) = [\alpha(x),\alpha(y)] = \alpha(x)\alpha(y)\alpha(x^{-1})\alpha(y^{-1}) = [\alpha,x]^{-1}x^{-1}[\alpha,y]^{-1}y^{-1}x[\alpha,x]y[\alpha,y] \]
\[ = xyx^{-1}y^{-1} = [x,y],\]
so $\alpha$ fixes every commutator in $G$.  Since $G'$ is generated topologically by the commutators in $G$, it follows that the action of $\alpha$ on $G'$ is trivial.\end{proof}

\begin{lem}\label{fratemblem}Let $S$ be a non-trivial finitely generated pro-$p$ group and let $G \in \Emb(S)$ be quasi-Frattini.  Then $S/\rO_p(G)$ acts faithfully on $\rE_p(G/\rO_p(G))$.  In particular, $G$ is $p$-separable if and only if $S \unlhd G$.  If $G \in \Emb(S)$ is Frattini, then $G/\rO_p(G)$ acts faithfully on $\rE_p(G/\rO_p(G))$.\end{lem}

\begin{proof}Let $K = \rO_p(G) \cap \Phi(S)$ and let $E =  \rE_p(G/\rO_p(G))$.  By Corollary \ref{tatecor} (i), $\rO_{p'}(G/K)=1$.  Thus $ \rF^*(G/K)$ is generated by $\rO_p(G)/K$ together with the components of $G/K$, and all components of $G/K$ have order divisible by $p$.  The centraliser of $ \rF^*(G/K)$ inside $G/K$ is $\rZ( \rF^*(G/K))$, which is a subgroup of $\rO_p(G/K)$ since $\rO_{p'}(G/K)=1$.  The action of $S$ on $\rO_p(G)/K$ is trivial, since $\rO_p(G)/K$ corresponds to $\rO_p(G)\Phi(S)/\Phi(S)$, which is a central factor of $S$ as $\Phi(S) \ge [S,S]$.  Thus the kernel of the action of $S/K$ on $\rE_p(G/K)$ is contained in $\rO_p(G)$.  Now $E$ corresponds to a quotient of the perfect group $\rE_p(G/K)$ by a central subgroup, so $S/\rO_p(G)$ acts faithfully on $E$ by Lemma~\ref{centauto}.  If $S$ is not normal in $G$, then $S/\rO_p(G)$ is non-trivial, so $E$ is also non-trivial, so $G$ is not $p$-separable.

If $\rO_p(G) \le \Phi(S)$, we have $K = \rO_p(G)$, so $ \rF^*(G/K) = \rE_p(G/K) =  E$, and $\rZ(\rE_p(G/K))=1$ so we have a faithful action of $G/\rO_p(G)$ on $E$.\end{proof}

\section{Extension theory}

Given a group $G$ acting on an abelian group $M$, write $\rH^n(G,M)$ for the $n$-th cohomology group of $G$ on $M$.

\begin{prop}\label{htwothm}Let $G$ be a finite group, and let $M$ be an abelian finite group on which $G$ acts.  Given an extension
\[\mcE = \{\xymatrix{1 \ar[r] & M \ar[r]^\alpha & E \ar[r]^\pi & G \ar[r] & 1}\}\]
of $M$ by $G$, obtain $t_\mcE$ as follows: 

Let $\tau$ be any function from $G$ to $E$ such that $\pi \tau = \id_G$.  Let $f: G \times G \rightarrow M$ be the function determined by the equation $\tau(x)\tau(y) = \tau(xy) \alpha(f(x,y))$.  Let $t_\mcE$ be the equivalence class of $f$ modulo $2$-coboundaries.

Then:
\vspace{-12pt}
\begin{enumerate}[(i)]
\item $f$ is a $2$-cocycle, any choice of $\tau$ gives the same $t_\mcE$, and $t_\mcE$ depends only on the equivalence class of the extension $\mcE$;
\item the map $\mcE \mapsto t_\mcE$ defines a bijection from the set of equivalence classes of extensions of $M$ by $G$ to $\rH^2(G,M)$;
\item $\mcE$ splits if and only if $t_\mcE = 0$.\end{enumerate}\end{prop}

\begin{proof}See \cite{Wil}, Lemmas 6.2.1. and 6.2.2.  (In fact, \cite{Wil} gives a proof for profinite groups in the context of profinite cohomology.)\end{proof}

\begin{prop}\label{coprimecohom}Let $M$ be a finite abelian group, and let $G$ be a finite group acting on $M$.  Suppose $H$ is a subgroup of $G$ for which $|G:H|$ is coprime to $|M|$.  Then for $n>0$, the restriction map $\rH^n(G,M) \rightarrow \rH^n(H,M)$ is injective.\end{prop}

\begin{proof}
 See \cite{Eve}, Proposition 4.2.5.
\end{proof}

\begin{proof}[Proof of Theorem \ref{extnthm}]We may regard $P$ as an open subgroup of $S$, and $S$ as a $p$-Sylow subgroup of both $G$ and $G^*$.  Define subgroups $P_i$ of $P$ by $P_1 = P$, and thereafter $P_{i+1} = [P_i,P]P^p_i$.  Then $P_i$ is an open characteristic subgroup of $P$ for all $i$.  Set $G_i = G/P_i$, set $G^*_i = G^*/P_i$, and set $M_i = P_i/P_{i+1}$.  Then for $i \geq 1$, we have extensions $\mcE_i$ and $\mcE^*_i$ of finite groups given by 
\[\mcE_i = \{\xymatrix{1 \ar[r] & M_i \ar[r] & G_{i+1} \ar[r] & G_i \ar[r] & 1 }\}\]
\[\mcE^*_i = \{\xymatrix{1 \ar[r] & M_i \ar[r] & G^*_{i+1} \ar[r] & G^*_i \ar[r] & 1 }\}\]
and by an inverse limit argument, it suffices to prove that these extensions are equivalent for all $i$.  By induction, we may assume that we have an isomorphism $\theta$ between $G_i$ and $G^*_i$; furthermore, the actions of $G_i$ and $G^*_i$ on $P_i/P_{i+1}$ are determined by the action of $K$ on $P_i/P_{i+1}$, which is in turn determined by the action of $K$ on $P/\Phi(P)$, by Lemma \ref{cinvlem} (ii).  Hence $\theta$ induces an isomorphism from $M_i$ as a $G_i$-module to $M_i$ as a $G^*_i$-module.  Now by Proposition \ref{htwothm}, the extensions $\mcE_i$ and $\mcE^*_i$ are both associated in a natural way to elements $t$ and $t^*$ say of $\rH^2(G_i,M_i)$, and the extensions are equivalent if and only if $t=t^*$.  However, both extensions have the common restriction
\[\xymatrix{1 \ar[r] & M_i \ar[r] & S_{i+1} \ar[r] & S_i \ar[r] & 1 },\]
where $S_i = S/P_i$.  This corresponds to the condition that $t^\rho = (t^*)^{\rho}$, where
$\xymatrix{\rH^2(G_i,M_i) \ar[r]^\rho & \rH^2(S_i,M_i)}$ is the natural restriction map.  But $S_i$ is a $p$-Sylow subgroup of $G_i$ and $M_i$ is a $p$-group, so by Proposition \ref{coprimecohom}, $\rho$ is injective.  Hence $t=t^*$ and so $\mcE_i$ and $\mcE^*_i$ are equivalent.\end{proof}

Corollary \ref{embsizecor} is immediate, given the fact that a finitely generated pro-$p$ group has only finitely many (normal) subgroups of any given finite index.

\section{The critical cases}

In this section, we establish `critical' subclasses of $\Emb(S)$, $\EmbLF(S)$ and $\EmbS(S)$ with more restricted structure, such that for a fixed finitely generated pro-$p$ group $S$, the class $\Emb(S)$, $\EmbLF(S)$ or $\EmbS(S)$ is infinite if and only if the corresponding critical subclass is infinite.

\begin{defn}Let $G$ be a $p'$-embedding of the finitely generated pro-$p$ group $S$ and write $P=\rO_p(G)$.  Define the subclasses $\CEmbA(S)$, $\CEmbC(S)$, $\CEmbLF(S)$ and $\CEmbL(S)$ of $\Emb(S)$ respectively as follows:

Let $G \in \CEmbA(S)$ if $G=SH$ such that $H$ is a non-trivial finite elementary abelian $q$-group (for $q$ a prime distinct from $p$), $HP/P$ is a minimal normal subgroup of $G/P$, $G = \rO^q(G)$ and $\rN_G(P \cap \Phi(S)) = S$.

Let $G \in \CEmbC(S)$ if $G=SH$ such that $H$ is a non-abelian finite $q$-group (for $q$ a prime distinct from $p$) that has no proper critical subgroups in the sense of Thompson (in particular, $H$ is critical in itself, so $\Phi(H) \le \rZ(H)$), $HP/\rZ(H)P$ is a chief factor of $G$, $G = \rO^q(G)$ and $\rN_G(P \cap \Phi(S)) \le S\rZ(H)$.  Define $\CEmbS(S) := \CEmbA(S) \cup \CEmbC(S)$.

Let $G \in \CEmbLF(S)$ if $\rE(G)=1$ and $G=SQ$ such that $Q \ge P$ and $Q/P$ is the normal closure of a component of $G/P$ of order divisible by $p$.

Let $G \in \CEmbL(S)$ if $G=SQ$ such that $Q$ is the normal closure of a component of $G$. (Here the component is necessarily of order divisible by $p$.)\end{defn}

\begin{thm}\label{critthm}Let $S$ be a finitely generated pro-$p$ group.
\begin{enumerate}[(i)]
\item If $\CEmbS(S)$ is finite then $\EmbS(S)$ is finite.
\item If $\CEmbS(S)$ and $\CEmbLF(S)$ are finite then $\EmbLF(S)$ is finite.
\item If $\CEmbS(S)$, $\CEmbLF(S)$ and $\CEmbL(S)$ are finite then $\Emb(S)$ is finite.\end{enumerate}\end{thm}

\begin{defn}Let $S$ be a finitely-generated pro-$p$ group.  Define the invariant $d_f(S)$ to be the maximum value of $\log_p|K\Phi(S):\Phi(S)|$ as $K$ ranges over the finite normal subgroups of $S$.  For instance, $d_f(S) = d(S)$ if and only if $S$ is finite, while $d_f(S)=0$ if and only if all finite normal subgroups of $S$ are contained in $\Phi(S)$.\end{defn}

\begin{lem}\label{pspace}Let $G$ be a profinite group with a finitely-generated $p$-Sylow subgroup $S$.  Let $\mcX$ be a set of finite normal subgroups of $G$ and let $H = \overline{\langle \mcX \rangle}$.  Then there is a subset $\mcK$ of $\mcX$ such that $|\mcK| \le \log_p|H\Phi(S):\Phi(S)|$ and such that $H/\langle \mcK \rangle$ has a normal $p$-complement.

In particular, if $\Omega$ is the set of components of $G$ of order divisible by $p$, then $S$ has at most $d_f(S)$ orbits on $\Omega$ (acting by conjugation).\end{lem}

\begin{proof}Given a normal subgroup $N$ of $G$, write $V_S(N) = (N \cap S)\Phi(S)/\Phi(S)$, regarded as a subspace of $S/\Phi(S) \cong (\bF_p)^{d(S)}$.  Since $H$ is generated by $\mcX$, there are $H_1, \dots, H_k \in \mcX$ such that
\[ V_S(H) = V_S(H_1) + \dots + V_S(H_k),\]
and such that $k \leq \dim(V_S(H)) =  \log_p|H\Phi(S):\Phi(S)|$.  Now set $\mcK = \{H_1, \dots, H_k\}$ and let $K = \langle \mcK \rangle$; then clearly 
\[ \Phi(S)(H \cap S) = \Phi(S)(K \cap S),\]
so $H/K$ has a normal $p$-complement by Corollary \ref{tatecor} (ii).

For the final assertion, let $H = \langle \Omega \rangle$.  Without loss of generality, we may assume $G = SH$; as $H$ is a central product of the elements of $\Omega$, the $S$-orbits on $\Omega$ are the same as $G$-orbits.  Indeed we have $H = \langle \mcX \rangle$, where $\mcX$ consists of the normal subgroups of $G$ formed by taking the product of the $S$-conjugates of a single element of $\Omega$.  Since no element of $\mcX$ is redundant in generating $H$ and $H$ has no $p$-separable images, we conclude that $|\mcX| \le d_f(S)$, so there are at most $d_f(S)$ orbits of $S$ on $\Omega$.\end{proof}

\begin{lem}\label{elabcrit}Let $P$ be a finite abelian $p$-group.  Write $\Omega_i(P)$ for the group of elements of $P$ of order dividing $p^i$.  Let $\alpha$ be a non-trivial automorphism of $P$ of order coprime to $p$.  Then $\alpha$ induces a non-trivial automorphism of $\Omega_1(P)$.\end{lem}

\begin{proof}Clearly $\Omega_1(P)$ is characteristic, so $\alpha$ induces an automorphism of $\Omega_1(P)$.  Let $G = P \rtimes \langle \alpha \rangle$.  Suppose that $\alpha$ fixes $\Omega_1(P)$ pointwise.  Let $p^{i+1}$ be the exponent of $P$, and let $x \in P$.  Then $x^{p^i} \in \Omega_1(P)$, so $\alpha(x)x^{-1}$ has order dividing $p^i$, since
\[(\alpha(x)x^{-1})^{p^i} = \alpha(x^{p^i})(x^{p^i})^{-1} = 1.\]
In other words, $[\langle \alpha \rangle,P] \le \Omega_i(P)$ and hence $[G,G,G] \le \Omega_i(P)$ since $G' \le P$.  Repeating the argument, we see that $G$ is nilpotent.  But then $G$ is the direct product of its Sylow subgroups, so $\alpha$ centralises $P$.\end{proof}

\begin{proof}[Proof of Theorem \ref{critthm}]Let $G$ be a $p'$-embedding of $S$.  In all cases we will obtain subgroups $L_1,\dots,L_k$ of $G$, each belonging to one of the classes $\CEmbS(S)$, $\CEmbLF(S)$ and $\CEmbL(S)$ (depending on whether $\rE(G)=1$ and/or $G$ is $p$-separable), such that $|G:S|$ is bounded by a function of $\max |L_i:S|$ and $S$.  The conclusion will then follow by Corollary \ref{embsizecor}.

Let $P = \rO_p(G)$ and let $P \le F \le G$ such that $F/P =  \rF^*(G/P)$.  Then the order of $G/P$, and thus the index $|G:S|$, is bounded by a function of $|F:P|$, since the generalised Fitting subgroup of $G/P$ contains its own centraliser.  In turn $|F:P|=|F:\rO_p(F)|$ is bounded by a function of the $p'$-order of $F$, which is $|FS:S|$.  Thus we may assume $G = FS$.

In this case $G$ is the (permutable) product of the subgroups $S, F_{p_1},\dots,F_{p_m},E_1,\dots,E_n$, with $p,p_1,\dots,p_m$ distinct primes, such that $F_{p_i}/P = \rO_{p_i}(G/P)$ and $E_j/P$ is the group generated by the $S/P$-conjugates of a component of $G/P$.  Moreover, $n$ is at most $d(S)$ by Lemma \ref{pspace}.

Let $H = SF_{p_i}$ for some $i$.  Then $H$ is prosoluble.  Moreover, we have $\rC_H(P) \le P$, because $\rC_G(P)/\rZ(P)$ acts faithfully on $\rE(G)$ by Lemma~\ref{gfitlem}, whereas $H$ centralises $\rE(G)$.  Thus $H \in \EmbS(S)$.

If $G$ is prosoluble then $n=0$.  Otherwise let $K = SE_j$ for some $j$.  Since $E_j$ is normal in $G$ we have $\rO_{p'}(E_j) = \rO_{p'}(G) = 1$, so $\rO_{p'}(K)=1$.  Thus $K \in \Emb(S)$.  Also, any component of $K$ is a component of $E_j$ and hence of $G$, so if $\rE(G)=1$ then $\rE(K)=1$.

Thus to obtain a bound on $|G:S|$, it suffices to bound the $p'$-order of each of the subgroups $SF_{p_i}$ and $SE_j$ individually.

Suppose $G=SE_1$.  If $\rE(G)>1$, then some and hence all components of $G/P$ arise from components of $G$, that is $G=S\rE(G)$.  Since the components of $G/P$ form a single $S$-orbit, the same is true for the components of $G$, so $G \in \CEmbL(S)$.  Suppose instead that $\rE(G)=1$ and $G$ is $p$-separable.  Then by the Frattini argument, for each prime $q$ dividing $|E_1/P|$ we can find a $q$-Sylow subgroup $H_q/P$ of $E_1/P$ that is normalised by $S/P$, and then to bound the $p'$-order of $G$, it suffices to bound the $p'$-orders of the groups $SH_q$ for all primes $q$.  Note that as $G$ is $p$-separable, we have $\rE(G)=1$, ensuring that $SH_q$ is a $p'$-embedding of $S$ by Corollary~\ref{lflem}.  Thus this situation reduces to considering prosoluble $p'$-embeddings, which in turn reduces to $p'$-embeddings of the form $G=SF_{p_i}$.  

The only remaining case of interest if $G = SE_1$ is if $\rE(G)=1$ and $p$ divides $|E_1/P|$, in which case $G \in \CEmbLF(S)$ by construction.

We have now reduced to the case $G=SF_q$, where $q=p_1$ is some prime distinct from $p$.

Let $\mcN$ be the class of $p$-separable $p'$-embeddings of $S$ in which $S$ is normal.  If $G \in \mcN$ then $|G:S|$ divides $|\GL(d(S),p)|$, so by Corollary \ref{embsizecor}, $\mcN$ is finite.  Let $G \in \Emb(S)$ and let $R = \rO_p(G) \cap \Phi(S)$.  Suppose that $G$ satisfies all the conditions for membership of the class $\CEmbA(S)$, except that $\rN_G(R) \not= S$.  Then $\rN_G(R) > S$, so in fact $\rN_G(R) = G$ by the irreducibility of the action of $S$ on $HP/P$.  Similarly, if $G$ satisfies all the conditions for membership of the class $\CEmbC(S)$ except that $\rN_G(R) \not\le S\rZ(H)$, then $R \unlhd G$ by the irreducibility of the action of $S$ on $HP/\rZ(H)P$.  Thus $G \in \mcN$ by Lemma \ref{fratemblem}.  Write $\CEmbA(S)' = \CEmbA(S) \cup \mcN$ and $\CEmbC(S)' = \CEmbC(S) \cup \mcN$.

Let $H$ be a $q$-Sylow subgroup of $G$ contained in $F_q$.  Then $H$ is a finite $q$-group and $PH$ is normal in $G$.  Our strategy is to bound $|S:P|$: this will produce a bound for $|G:P|$, because $G/P$ acts faithfully on $P/\Phi(P)$, and by the Schreier index formula we have $d(P) \le |S:P|(d(S)-1)+1$.  Hence can freely replace $G$ with a subgroup $L$ of $G$ containing $S$ such that $\rO_p(L)=P$, or in other words $L=SH_0$ where $H_0$ is a subgroup of $H$ such that $S/P$ acts faithfully on $H_0P/P$.  Thus we may assume $G = \rO^q(G)$, since $\rO^q(G)$ is normal in $G$ and contains $S$.  By Theorem~\ref{thomcrit}, we may assume $H$ is critical in itself; otherwise we could replace $H$ by a critical subgroup without changing $\rO_p(G)$.  If $H$ is abelian, we can replace $H$ by $\Omega_1(H)$, by Lemma \ref{elabcrit}, and so assume $H$ is elementary abelian.

Let $M = HP/P$ if $H$ is abelian and let $M = HP/\rZ(H)P$ otherwise.  Then $M$ is a module for $S$ over the field of $q$ elements.  By a version of Maschke's theorem, $M$ is completely reducible.

Suppose $H$ is abelian.  Then we can write $H = H_1 \times \dots \times H_n$ such that for each $i$, $PH_i$ is a minimal normal subgroup of $SH_i$, and thus $SH_i \in \CEmbA(S)'$.  Let $P_i = \rO_p(SH_i)$.  Suppose now that $\CEmbA(S)$ is finite.  Then there are only finitely many possibilities for $P_i$ as a subgroup of $S$; thus there are only finitely many possibilities for $P = \bigcap^n_{i=1} P_i$.  We see from this that there are only finitely many $p'$-embeddings of $S$ of the form $SK$ where $K$ is abelian.  

Suppose now that $H$ is non-abelian.  Then we can write $H = H_1\dots H_n$ such that $H_i \cap H_j = \rZ(H)$ for $i$ and $j$ distinct, and so that $PH_i/P\rZ(H)$ is a chief factor of $SH_i$.  Again we set $P_i = \rO_p(SH_i)$ and note that $P = \bigcap^n_{i=1}H_i$.  If $H_i$ is non-abelian, this implies $SH_i \in \CEmbC(S)'$, while if $H_i$ is abelian, the finiteness of $\CEmbA(S)$ leaves only finitely many possibilities for $P_i$.  Thus if $\CEmbS(S)$ is finite, there are only finitely many possibilities for $P$ and hence for $G$.

The above argument shows that if $\CEmbS(S)$, $\CEmbL(S)$ and $\CEmbLF(S)$ are all finite, then $\Emb(S)$ is finite.  Note, moreover, that if $G$ is in the class $\EmbS(S)$, then $|G:S|$ is in fact bounded using groups in $\CEmbS(S)$ only, while if $G$ is in the class $\EmbLF(S)$, the groups in $\CEmbS(S) \cup \CEmbLF(S)$ suffice.  This demonstrates all three assertions in the theorem.\end{proof}

\section{Profinite groups with a cyclic or $2$-generator Sylow subgroup}

For this section, $S$ is a pro-$p$ group such that $d(S) \le 2$.  The significance of this condition (in light of Lemma \ref{fratemblem}) is that if $G$ is a $p'$-embedding of $S$, then either $S/\rO_p(G)$ is cyclic, or else $G$ is a Frattini $p'$-embedding and thus has a special structure.

First, consider the case that $S$ is (topologically) cyclic, that is $d(S)=1$.  Here the possibilities are very straightforward:

\begin{prop}Let $S$ be a cyclic pro-$p$ group, and let $G \in \Emb(S)$.  Then exactly one of the following holds:
\begin{enumerate}[(i)]
\item $S \unlhd G$ and $G/S$ is cyclic of order dividing $p-1$;
\item $S$ is finite and $G$ has a single component $Q$, such that $S \leq Q$ and $G/\rZ(Q)$ is almost simple.
\end{enumerate}
\end{prop}

\begin{proof}Let $P = \rO_p(G)$.  If $S=P$, then case (i) occurs.  Otherwise $P \leq \Phi(S)$, so $G/P$ acts faithfully on $\rE_p(G/P)$ by Lemma \ref{fratemblem}.  Let $R/P$ be a component of $G/P$.  Then $R$ is a central extension of $P$ by $R/P$, since $\Aut(P)$ is $p$-separable, so there is a component $Q$ of $G$ such that $R = PQ$.  Since $Q \unlhd G$ but $Q$ is not $p$-separable, we have $S \cap Q \not\le \Phi(S)$ by Corollary \ref{tatecor}, so $S \le Q$.  Clearly now $Q = \rE_p(G) = \rE(G)$, and $G/\rZ(Q)$ is almost simple, since $G/P = G/\rZ(Q)$ acts faithfully on $Q/\rZ(Q)$.\end{proof}

We now obtain a list of possible structures for $p'$-embeddings of a $2$-generator pro-$p$ group.

\begin{thm}\label{2genthm}Let $S$ be a pro-$p$ group such that $d(S) = 2$, and let $G \in \Emb(S)$.  Write $P= \rO_p(G)$ and $H=G/\rO_p(G)$.
  
If $G$ is a standard $p'$-embedding, then exactly one of the following holds:
 
\begin{enumerate}[(i)]
\item $p$ is odd, $S/P$ is non-trivial cyclic and there is a quasisimple normal subgroup $Q$ of $H$ such that $S/P$ is a $p$-Sylow subgroup of $Q$;
\item $S/P$ is non-trivial cyclic, $H$ acts faithfully on $P/\Phi(P)$ and $| \rF^*(H)|$ is coprime to $p$.
\setcounter{saveenum}{\value{enumi}}
\end{enumerate}

If $G$ is a quasi-Frattini but not Frattini $p'$-embedding, then either (i) holds or the following holds:
\begin{enumerate}[(i)]
\setcounter{enumi}{\value{saveenum}}
\item $S = P$ and $H$ is isomorphic to a $p'$-subgroup of $\GL(2,p)$.
\setcounter{saveenum}{\value{enumi}}
\end{enumerate}
 
If instead $G$ is a Frattini $p'$-embedding, then ${\rC_H(\rE_p(H))=1}$ (so in particular ${\rE(H) = \rE_p(H)}$) and exactly one of the following holds:
\begin{enumerate}[(i)]
\setcounter{enumi}{\value{saveenum}}
\item There is a subgroup $Q$ of $G$ containing $S$ such that $Q/P$ is a non-abelian simple group with a $2$-generator $p$-Sylow subgroup;
\item $p$ is odd and there is a subgroup $Q$ of $G$ containing $S$ such that $Q/P$ is a direct product of two non-abelian finite simple groups (possibly isomorphic), each having a non-trivial cyclic $p$-Sylow subgroup;
\item $\rE(H)$ is the direct product of $p^l$ copies of a single non-abelian finite simple subgroup $Q$ of $H$ for some $l \ge 0$, with $\rE(H)$ being the $S$-invariant closure of $Q$, and $H/\rE(H)$ has a non-trivial cyclic $p$-Sylow subgroup.
\end{enumerate}
\end{thm}

\begin{proof}Let $k = |S:P\Phi(S)|$.  Since $d(S)=2$, we have $k \in \{1,p,p^2\}$.

If $k=1$, then $S=P$ and we are clearly in case (iii) by Corollary~\ref{lflem}.  A $p'$-embedding with $S=P$ is evidently quasi-Frattini but not Frattini.

If $k=p$, then $S/P$ is non-trivial cyclic.  If $| \rF^*(H)|$ is coprime to $p$, we see that $\rE(G)=1$ since every component of $G$ must have order divisible by $p$, so $H$ acts faithfully on $P/\Phi(P)$ by Corollary \ref{lflem} and we are in case (ii).  In case (ii), $G$ is $p$-separable and therefore a standard $p'$-embedding by Lemma~\ref{fratemblem}.  If instead $p$ divides $ \rF^*(H)$, then there is some quasisimple subgroup $Q$ of $H$ of order divisible by $p$; this ensures that $|Q/\rZ(Q)|$ is also divisible by $p$.  Let $K$ be the normal closure of $Q$ in $H$.  Then $K \ge S/P$, since otherwise  we would have $K \cap S/P \le \Phi(S/P)$, which would imply that $K$ has a normal $p$-complement by Corollary~\ref{tatecor}.  Moreover, $K$ is a central product of copies of $Q$; since the $p$-Sylow subgroup of $K$ is cyclic, there is only room for one copy of $Q$, in other words $K=Q$.  We see that $p$ is odd because there are no non-abelian finite simple groups with cyclic $2$-Sylow subgroups (see for instance exercise 262 of \cite{Ros}).  Thus we are in case (i).

We may now assume $k=p^2$, in other words, $G$ is a Frattini $p'$-embedding.  We have $\rC_H(\rE_p(H))=1$ by Lemma \ref{fratemblem}.  To simplify notation, let us divide out by $P$; in other words, assume that $P=1$ (so $G=H$) and $S$ is finite.

Suppose $\rE_p(G) \ge S$.  By Corollary~\ref{tatecor} applied to $\rE_p(G)$, every component $Q$ of $\rE_p(G)$ satisfies $Q \cap S \not\le \Phi(S)K$, where $K$ is the product of the other components.  This leaves only two possibilities: either $\rE_p(G)$ is a non-abelian simple group $Q$ with a $2$-generator $p$-Sylow subgroup, or $\rE_p(G) = Q_1 \times Q_2$, where $Q_1$ and $Q_2$ are non-abelian simple groups with cyclic $p$-Sylow subgroups (here $p$ is necessarily odd).  These are cases (iv) and (v) respectively.

Finally, suppose $\rE_p(G) \not\ge S$.  We cannot have $\rE_p(G) \cap S \le \Phi(S)$, so $\Phi(S)$ has index $p$ in $\Phi(S)(\rE_p(G) \cap S)$.  By Lemma~\ref{pspace}, we see that $\rE_p(G)$ is the $S$-invariant closure of a single component $Q$, in other words $\rE_p(G)$ is the direct product of the $S$-conjugates of $Q$, whose number is a power of $p$ as $S$ is a pro-$p$ group.  Since $|S:\Phi(S)(\rE_p(G) \cap S)|=p$, the $p$-Sylow subgroup of $G/\rE_p(G)$ is non-trivial cyclic.  This is case (vi).\end{proof}

\begin{rem}(a) Only cases (ii) and (iii) can give rise to $p$-separable $p'$-embeddings, and case (iii) accounts for only finitely many $p'$-embeddings.  In cases (i), (iv) and (v), the isomorphism type of the simple group $Q/\rZ(Q)$ involved in $\rE(G/\rO_p(G))$ is restricted (see Lemma~\ref{dpdeg}), while in each case a bound on the order of $Q$ would imply a bound on the index $|G:S|$.  Thus in cases (i), (iv) and (v), the possibility of infinitely many $p'$-embeddings remains only because of the existence of infinitely many finite simple groups of Lie type of small rank (obtained by varying the field of definition).

(b) If $S$ is infinite and not finite-by-$\bZ_p$, then every finite normal subgroup of $S$ is contained in $\Phi(S)$, so $\rE(G)= 1$ for all $p'$-embeddings $G$ of $S$ by Corollary~\ref{tatecor}.\end{rem}

We now give a construction to demonstrate Proposition \ref{ascembed}.

\begin{eg}Let $p$ and $q$ be primes.  Let $\bF = \bF_{p^q}$ and let $\theta$ be the Frobenius automorphism of $\bF$.  Let $K$ be the set of clopen subsets of $\bZ_p$.  Let $F$ be the (elementary abelian) group of additive functions from $K$ to $\bF$, that is, functions $f:K \rightarrow \bF$ such that $f(u \cup v) = f(u) + f(v)$ whenever $u$ and $v$ are disjoint. Let $Z \cong \bZ_p$ act on $F$ by translating the elements of the domain, giving a semidirect product $S = F \rtimes Z$.  We claim that $S$ is a $2$-generator metabelian pro-$p$ group; indeed it is the inverse limit of the $2$-generator metabelian $p$-groups $F_n \rtimes \bZ_p/p^n\bZ_p$, where $F_n$ is the group of functions from $\bZ_p/p^n\bZ_p$ to $\bF$.  There is a natural surjective map $\phi_n: F \rightarrow F_n$ formed by restricting the domain, and then maps $F \rtimes \bZ_p \rightarrow F_n \rtimes \bZ_p/p^n\bZ_p$ are given by extending $\phi_n$ in a way that is compatible with the action of $Z$ on $F$.

The group $G$ is formed as $F \rtimes (Q \rtimes Z)$, equipped with the topology in which $F \rtimes Z$ is an open compact subgroup, where $Q$ is a subgroup of $\Aut(F)$ of the form $\bigcup_{i \in \bN} Q_i$.  As a group of automorphisms of $F$, the group $Q_i$ has the following description: $Q_i$ is a direct product of copies of $C_q$ indexed by the elements of $\bZ_p/p^i\bZ_p$, and a generator for the $j$-th copy of $C_q$ in $Q_i$ acts on $F$ by replacing $f(u)$ by $(f(u))^\theta$ for all $f \in F$ and all $u \in K$ such that $u \subseteq j$, with the consequent alteration of $f(u)$ in the more general case that $u \cap j \not= \emptyset$.  (Note that $j$ is a coset of $p^i\bZ_p$, being an element of $\bZ_p/p^i\bZ_p$).  It is easily verified that as subgroups of $\Aut(F)$, $Q_i$ is normalised by $Z$ and $Q_i < Q_{i+1}$ for all $i$.  Thus the groups
\[ G_i = F \rtimes (Q_i \rtimes Z) \]
for $i \ge 0$ form an ascending chain of subgroups of $G$, each open in the next, with union $G$.  Given any finite image $R$ of $G_i$, and given a conjugacy class $C$ of $R$, then $|C| = p^aq^b$ where $b$ is at most $p^i$.  Moreover, for a sufficiently large finite image, there is a conjugacy class contained in the image of $F$ whose size is divisible by $q^{p^i}$: let $\alpha \in \bF$ be primitive, let $f_i \in F$ be given by $f_i(U) = |U \cap \{0,1,\dots,p^i-1\}|\alpha$, and consider the conjugacy class of the image of $f_i$ in a sufficiently large finite quotient of $G_i$.  Thus the fusion of conjugacy classes of $S$ in $G_i$ and $G_j$ is inequivalent for $i \not= j$, even up to automorphisms of $S$.

For $p$ and $q$ distinct primes, it is clear that this construction satisfies all assertions in Proposition \ref{ascembed}.

In the construction, we notice totally disconnected, locally compact groups with a further interesting property.   Let $R = Q \rtimes Z \cong G/F$.  Let $U$ be an open compact subgroup of $R$.  We claim that $\rN_R(U)/U$ is finite, and indeed that $R$ acts properly by conjugation on the metric space of open compact subgroups of $R$ with metric given by
\[d(U,V) =  \log(|U: U \cap V||V: U \cap V|).\]
To prove that the action on the above metric space is proper, it suffices to show that the set $\{r \in R \mid |U: U \cap U^r| \le p^k\}$ is compact for all $k$ and fixed $U$, so we are free to take $U = Z$.  In this case the set $R_k = \{r \in R \mid |U: U \cap U^r| \le p^k\}$ decomposes as $R_k = (R_k \cap Q)Z$.  Now $Z$ is compact, and $R_k \cap Q$ is precisely the finite group $\rC_Q(p^kU) = Q_k$.  Thus $R_k$ is compact as required.

Note that the construction is valid even if $p=q$, in which case we obtain a metabelian totally disconnected, locally compact group $R$ that is the union of an ascending chain of open pro-$p$ subgroups, such that every open compact subgroup of $R$ has finite index in its normaliser.\end{eg}

\section{Normal subgroup conditions}

\begin{lem}\label{obphilem}Let $S$ be a finitely generated pro-$p$ group and let $N$ be an open normal subgroup of $S$.  Let $\mcK$ be the set of open normal subgroups of $S$ that are not contained in $N$.  The following are equivalent:
\begin{enumerate}[(i)]
\item $\mcK$ is finite;
\item $N$ contains every normal subgroup of $S$ of infinite index.
\end{enumerate}
\end{lem}

\begin{proof}Suppose there is a normal subgroup $P$ of $S$ of infinite index that is not contained in $N$.  Then $P$ is the intersection of a descending chain $P_1 > P_2 > \dots$ of open normal subgroups of $S$, none of which are contained in $N$.  Thus $\mcK$ is infinite.

Conversely, suppose $\mcK$ is infinite.  We construct a directed graph $\Gamma$ on $\mcK$ by drawing an edge $(K_1,K_2)$ if $K_1 > K_2$ and $K_1/K_2$ is a chief factor of $S$.  Then every vertex lies on a path from the vertex $S$; moreover, $K/\Phi(K)$ is finite for every $K \in \mcK$ since $S$ is finitely generated, so $\Gamma$ is locally finite.  Thus $\Gamma$ contains an infinite path by K\H{o}nig's lemma, so there is an infinite descending chain $L_1 > L_2 > \dots$ in $\mcK$.  By a standard compactness argument, the intersection of the $L_i$ is a normal subgroup $L$ which is not contained in $N$, but $L$ has infinite index.\end{proof}

\begin{defn}Let $G$ be a finite simple group.  Define $\deg(G)$ to be the smallest number $d$ such that $G$ is isomorphic to a subgroup of $\GL(F^d)$ for some field $F$.

Given a profinite group $G$ and a prime $p$, define $d_p(G)$ to be $d(S)$ where $S$ is a $p$-Sylow subgroup of $G$.\end{defn}

\begin{lem}\label{dpdeg}Let $p$ be a prime and let $d$ be an integer.  Then there some integer $c$ depending on $d$ and $p$ such that if $G$ is a finite simple group such that $\deg(G) \ge c$, then $d_p(G) \ge d$.\end{lem}

\begin{proof}See \cite{ReiPhD}, section 1.7.\end{proof}

\begin{proof}[Proof of Theorem \ref{obphithm}]Let $t$ be an integer such that $d(S) - 1 \le t$, and also $|S:K| \le p^t$ for all $K \in \mcK$.  By Lemma \ref{obphilem}, every finite normal subgroup of $S$ is contained in $\Phi(S)$.  Thus $\Emb(S) = \EmbLF(S)$ by Corollary~\ref{finlfcor}.

Let $G$ be a $p'$-embedding of $S$, let $P=\rO_p(G)$, and let $E$ be such that $E/P = \rE_p(G/P)$.

Suppose $G$ is not a Frattini $p'$-embedding.  Then $P \in \mcK$, so $d(P) \le tp^t + 1$ by the Schreier index formula.  Since $G/P$ acts faithfully on $P/\Phi(P)$ (by Corollary \ref{lflem}), the index $|G:P|$ is bounded, leaving only finitely many possibilities for $G$ by Corollary \ref{embsizecor}.  In particular, this accounts for all prosoluble $p'$-embeddings, so $\EmbS(S)$ is finite.

Now suppose $|S:S^{(n)}|$ is finite for all $n$.  By the previous argument, we may now assume $G$ is a Frattini $p'$-embedding; this ensures that $G/P$ acts faithfully on $E/P$ by Lemma \ref{fratemblem}.  We proceed by a series of claims.

\emph{(i) We have $d_p(Q) \le tp^t + 1$ for every component $Q$ of $G/P$.}

By Corollary \ref{tatefin}, we have $E \cap S \not\le \Phi(S)$, so $E \cap S \in \mcK$, and hence $d(E \cap S) \le tp^t+1$ by the Schreier index formula; note that $E \cap S$ is a $p$-Sylow subgroups of $E$.  In turn, the direct decomposition of $E/P$ ensures that $d_p(Q) \le d(E \cap S)$.

\emph{(ii) Let $T$ be a $p$-Sylow subgroup of $E/P$ contained in $S/P$.  Then the derived length $l$ of $T$ is bounded by a function of $p$ and $t$.}

Let $Q$ be a simple direct factor of $E/P$.  It follows from claim (i) and Lemma \ref{dpdeg} that $\deg(Q)$ is bounded by a function of $p$ and $t$, so in particular $Q$ has a faithful linear representation of bounded degree.  Thus, by a theorem of Zassenhaus (\cite{Zas}), the derived length of any soluble subgroup of $Q$ is bounded by a function of $p$ and $t$.  Since $E/P$ is the direct product of its simple factors, the same bound applies to the derived length of $T$.

\emph{(iii) There is a bound on $|S:P|$ in terms of properties of $S$.}

Let $R = S/P$.  We already know that $|S:E \cap S|$ is at most $p^t$, so $T$ contains $R^{(t)}$.  But then $R^{(t+l)} \le T^{(l)} = 1$, so $S/P$ is soluble of derived length at most $t+l$.  This means that $P$ contains the open subgroup $S^{(t+l)}$, so $|S:P|$ is bounded by properties of $S$.

\emph{(iv) There is a bound on $|G:P|$ in terms of properties of $S$.}

We have a bound on $|S:P|$, giving a bound on $d(P)$ in terms of properties of $S$.  But $\rE(G)=1$, so $G/P$ is isomorphic to a subgroup of $\GL(d(P),p)$ by Corollary \ref{lflem}.

We conclude from claim (iv) and Corollary \ref{embsizecor} that $\Emb(S)$ is finite.\end{proof}

\begin{proof}[Proof of Theorem \ref{jiemb}]Let $\mcK$ be as in Theorem \ref{obphithm}.  Then $\mcK$ is finite by Lemma \ref{obphilem}.  If $S$ is insoluble, then $\Emb(S)$ is finite by Theorem \ref{obphithm}.  If $S$ is soluble, then the last non-trivial term in its derived series has finite index, so $S$ is virtually abelian.  In this case $S$ has finite subgroup rank $r$ say.  As a consequence, given any $p'$-embedding $G$ of $S$, then $G/\rO_p(G)$ is isomorphic to a subgroup of $\GL(r,p)$ by Corollary \ref{lflem} (since $d(P) \le r$ and $\rE(G)=1$), so $|G:S|$ is bounded by a function of $p$ and $r$, and thus $\Emb(S)$ is finite by Corollary \ref{embsizecor}.

For the final assertion, note that the just infinite groups $G$ having $S$ as a Sylow subgroup are precisely the $p'$-embeddings of $S$: since $S$ is infinite, any just infinite profinite group $G$ having $S$ as a Sylow subgroup must have $\rO_{p'}(G)=1$, and conversely any profinite group $G$ with $S$ as a Sylow subgroup and $\rO_{p'}(G)=1$ cannot have any finite normal subgroups, so $G$ is just infinite by \cite{ReiFI}, Lemma 4.\end{proof}

\section{Weakly regular pro-$p$ groups}

\begin{defn}Let $S$ be a finitely generated pro-$p$ group.  Say $S$ is \emph{weakly regular} if there does not exist a surjective homomorphism $S \rightarrow C_p \wr C_p$.\end{defn}

\begin{thm}[Yoshida \cite{Yosh} (finite version); Gilotti, Ribes, Serena \cite{GRS} (profinite version)]\label{yosh}Let $G$ be a profinite group and let $S$ be a $p$-Sylow subgroup of $G$.  Suppose $S$ is weakly regular.  Then $\rN_G(S)$ controls $p$-transfer in $G$.\end{thm}

As a consequence, we obtain significant restrictions on the structure of $p'$-embeddings of a weakly regular pro-$p$ group.

Given distinct primes $p$ and $q$, write $\ord(p,q)$ for the least positive integer $a$ such that $p^a \equiv 1 \mod q$.  Note that the elementary abelian group of order $p^d$ has an automorphism of order $q$ if and only if $\ord(p,q) \le d$ (using the formula for the order of the general linear group).

\begin{thm}\label{wreg}Let $S$ be a weakly regular pro-$p$ group and let $G \in \Emb(S)$.
\begin{enumerate}[(i)]
\item Suppose $G$ is of the form $G=SH$ where $H$ is abelian and $\rO_p(G)H$ is normal in $G$.  Then $S \unlhd G$.  Consequently $\CEmbA(S) = \emptyset$.
\item Let $G \in \EmbS(S)$ and let $q$ be a prime divisor of $|G:S|$.  Then $S$ has an automorphism of order $q$, so in particular $\ord(p,q) \le d(S)$.  If $q$ divides $|G:\rN_G(S)|$, then the following additional conditions are satisfied:
\begin{enumerate}[(a)]
\item $S$ has an automorphism of order $q$ that acts reducibly on $S/\Phi(S)$, so in particular $\ord(p,q) < d(S)$;
\item If $p$ is odd, then $\ord(q,p)$ is even.\end{enumerate}
\item Let $K$ be a normal subgroup of $G$ such that $K \le S$, and let $Q/K$ be a component of $G/K$ of order divisible by $p$.  Then $S$ normalises $Q$.  In particular, if $G \in \CEmbLF(S) \cup \CEmbL(S)$, then $G$ has exactly one non-abelian composition factor.\end{enumerate}\end{thm}

Theorem~\ref{wreg} will be proved at the end of this section.

\begin{eg}Given $d(S)$ and $p$, let $\pi$ be the set of primes satisfying the conditions in Theorem~\ref{wreg} (ii).  For some values of $d(S)$ and $p$, the set $\pi$ is surprisingly small.  For instance, suppose $p=3$, and $d(S) \leq 11$.  Then $\pi = \{2,5,11,41\}$.  So if $S$ is a weakly regular pro-$3$ group generated by at most $11$ elements, and $G$ is a $3$-separable $3'$-embedding of $S$, then the prime divisors of $|G:\rN_G(S)|$ are a subset of $\{2,5,11,41\}$.  Similarly, if $p=7$ and $d(S) \leq 8$, then $\pi \subseteq \{3,5,19\}$.\end{eg}

\begin{lem}\label{selftranslem}Let $S$ be a pro-$p$ group and let $G \in \EmbLF(S)$.  Let $K$ be a subgroup of $G$ that properly contains $S$.
\begin{enumerate}[(i)]
\item $S$ does not control $p$-transfer in $K$.
\item Suppose $S$ is weakly regular.  Then $\rN_K(S) > S$.\end{enumerate}\end{lem}

\begin{proof}(i) Suppose $S$ controls $p$-transfer in $K$.  Then by Theorem~\ref{tate}, $\rO^p(K)=\rO_{p'}(K)$ is a complement to $S$ in $K$.  But by Corollary~\ref{lflem} $\rO_{p'}(K)=1$, so $K$ is a pro-$p$ group, which is impossible since $S$ is a maximal pro-$p$ subgroup of $G$.

(ii) This follows immediately from part (i) together with Theorem~\ref{yosh}.\end{proof}

\begin{prop}\label{abecent}Let $S$ be a weakly regular pro-$p$ group, and let $G \in \EmbLF(S)$.  Let $H=S[G,S]$, and let $P=\rO_p(G)$.  Then:
\begin{enumerate}[(i)]
\item Any abelian $p'$-subgroup of $G/P$ that is normalised by $H/P$ is centralised by $H/P$;
\item $\rF(H/P)$ has nilpotency class at most $2$.
\end{enumerate}
\end{prop}

\begin{proof}(i) It suffices to consider abelian $q$-subgroups of $G/P$, where $q \in p'$.  Let $K \leq G$ such that $K'\rO^{q}(K) \leq \rO_p(G)$ and $[K,H] \leq \rO_p(G)K$; it is clear that this accounts for all abelian $q$-subgroups of $G/P$ that are normalised by $H/P$.  Then $\rN_{K/P}(S/P) = \rC_{K/P}(S/P)$, and $[K/P,S/P] \cap \rC_{K/P}(S/P) =1$ by part (iii) of Lemma~\ref{cinvlem}.  Let $M = S[K,S]$.  Since $P \leq S$, it follows that $\rN_M(S) = S$.  Hence $M=S$ by Lemma \ref{selftranslem}, so $[K,S] \leq K \cap S \leq P$.  The same argument shows that $K/P$ commutes with every $p$-Sylow subgroup of $G/P$.  But $H/P$ is generated by these $p$-Sylow subgroups by construction, so $K/P$ is centralised by $H/P$.

(ii) Write $T = \rF(H/P)$.  Since $H/P$ is finite, $T$ is nilpotent.  Let $c$ be the nilpotency class of $T$, and assume $c > 2$.  Then $\gamma_{c-1}(T)$ is abelian, since $[\gamma_{c-1}(T),\gamma_{c-1}(T)] \leq \gamma_{2c-2}(T)$, and $2c-2 = c + (c-2) > c$; thus $\gamma_{c-1}(T)$ is central in $T$ by part (i).  But then $\gamma_c(T)=1$, contradicting the definition of $c$.\end{proof}

\begin{cor}\label{fnilpcor}Let $S$ be a weakly regular pro-$p$ group, and let $G$ be a prosoluble $p'$-embedding of $S$.  Let $H=S[G,S]$, and let $P=\rO_p(H)$.  Then either $G$ is $p$-normal, or $\rF(H/P)$ has nilpotency class exactly $2$.\end{cor}

\begin{proof}By Proposition~\ref{abecent}, $\rF(H/P)$ has nilpotency class at most $2$, and clearly $H=P$ if $G$ is $p$-normal; hence we may assume $\rF(H/P)$ has nilpotency class less than $2$.  This means $\rF(H/P)$ is abelian, and so by the proposition $\rF(H/P)=\rZ(H/P)$.  Now $H/P$ is a finite soluble group, so $\rF(H/P) \geq \rC_{H/P}(\rF(H/P)) = H/P$, so $H/P$ is abelian, which means $S$ is normal in $H$.  By Sylow's theorem, $S$ is the unique $p$-Sylow subgroup of $H$.  But $H$ is generated by its $p$-Sylow subgroups.  Hence $H=S$, which means that $G$ is $p$-normal.\end{proof}

\begin{lem}\label{classicwr}Let $p$ be a prime, and let $q$ be a prime power coprime to $p$.  Let $n$ be any positive integer.  Suppose $p$ is odd, and let $G = \Sp(2n,q)$, considered as a subgroup of $\GL(V)$ where $V = \bF^{2n}_q$.  Suppose a $p$-Sylow subgroup of $G$ acts irreducibly on $V$.  Then $\ord(q,p)$ is even.\end{lem}

\begin{proof}See Table 1 of \cite{Sta}.  The Sylow subgroups of `type B' in this table are necessarily reducible.\end{proof}

\begin{lem}\label{sympauto}Let $q$ be an odd prime, and let $U$ be a $q$-group of nilpotency class $2$.  Let $P$ be a $p$-group of automorphisms of $U$, where $p \not= q$, such that $P$ centralises $\rZ(U)$.  Suppose also that $M = U/\rZ(U)$ is irreducible as a $P$-module.  Let $N$ be a maximal subgroup of $U'$, and identify $U'/N$ with $\bF_q$.  Then the homomorphism $(-,-)_N$ from $M \times M$ to $U'/N$ defined by $(x\rZ(U),y\rZ(U))_N = [x,y]N$ is a non-degenerate, skew-symmetric, alternating bilinear form for $M$ as a vector space over $\bF_q$, and this form is preserved by $P$.  Hence $P$ acts on $M$ as a subgroup of $\Sp(M)$, the symplectic group on $M$ associated to the given form.  In particular, $p \cdot \ord(q,p)$ is even.\end{lem}

\begin{proof}The equation $(x\rZ(U),y\rZ(U))_1 = [x,y]$ specifies a function $(-,-)_1$ from $M \times M$ to $U'$.  This is a homomorphism since $M$ is abelian, and hence it is surjective by the definition of $U'$; hence $(-,-)_N$ is a non-trivial quadratic form.  The form is preserved by $P$ since $P$ centralises $\rZ(U)$, which contains $U'$, and $M$ is irreducible as a $P$-module, so $(-,-)_N$ is non-degenerate on $M$.  Finally, $(-,-)_N$ is also skew-symmetric and alternating, since $[x,y] = [y,x]^{-1}$ and $[x,x] = 1$ are identities in any group.

We conclude that $P$ acts on $M$ as a subgroup of $\Sp(M)$.  Hence $\Sp(M)$ has a non-trivial irreducible $p$-subgroup.  This implies at least one of $p$ and $\ord(q,p)$ is even, by Lemma~\ref{classicwr}.\end{proof}

\begin{proof}[Proof of Theorem~\ref{wreg}]
(i) Let $P = \rO_p(G)$.  In this case, we see from Proposition~\ref{abecent} that $HP/P$ is central in $S[G,S]/P$, which implies that $S$ is normal in $S[G,S]$.  Since $S[G,S]$ is normal in $G$, it follows by Sylow's theorem that $S$ is normal in $G$.

(ii) Let $q$ be a prime divisor of $|G:S|$.  Then $q$ divides at least one of $|G:\rN_G(S)|$ and $|\rN_G(S):S|$.  If $q$ divides $|\rN_G(S):S|$, then there is an automorphism of $S$ of order $q$ induced by conjugation in $\rN_G(S)$, since $\rC_G(S) \le S$, and hence $\ord(p,q) \le d(S)$ by Lemma~\ref{cinvlem}.  So from now on we may assume $q$ divides $|G:\rN_G(S)|$.

Let $G_0 = 1$ and thereafter let $G_{i+1}$ be such that $G_{i+1}/G_i = \rO_p(G/G_i) \times \rO_{p'}(G/G_i)$.  We obtain a series
\[ G_1 < \dots < G_n = G\]
of open normal subgroups of $G$, where for all $i \ge 0$, the quotient $G_{i+1}/G_i$ is a pro-$p$ group if $i$ is even and a $p'$-group if $i$ is odd.  Set $H_i = G_{2i+1}$.  By the Frattini argument, for each index $i$ there is a $q$-Sylow subgroup $T_i/H_i$ of $\rO_{p'}(G/H_i)$ that is normalised by $S$.  The condition that $q$ divides $|G:\rN_G(S)|$ ensures that there is some $j \ge 0$ such that $S$ does not centralise $T_j/H_j$.  Now let $R = SG_{2j}/G_{2j}$ and consider the group $H = ST_j/G_{2j}$.  We see that $G/G_{2j} \in \EmbLF(R)$, so by Corollary \ref{lflem} we have $H \in \EmbLF(R)$; indeed $H \in \EmbS(R)$ since $H$ is $p$-separable.  Moreover, $R$ is weakly regular and $q$ divides $|R:\rN_H(R)|$, since $R$ does not normalise $S$.  Thus we may assume $G = ST$, where $T$ is a finite $q$-group, and that $T\rO_p(G)/\rO_p(G)$ is normal in $G/\rO_p(G)$.  Indeed, by Theorem~\ref{thomcrit}, we can find a characteristic critical subgroup $U$ of $T$ such that $S$ does not centralise $U$, and replacing $G$ with $S[G,S] = \rO^q(G)$ has no effect on the prime divisors of $|G:\rN_G(S)|$, since $G = S[G,S]\rN_G(S)$ by the Frattini argument.  The case $G \in \CEmbA(S)$ was already eliminated in part (i).  So we may assume $T$ is non-abelian, with no proper critical subgroups, so $T/\rZ(T)$ is elementary abelian.  Furthermore, we can replace $T$ with a subgroup $U > \rZ(T)$ such that $U\rO_p(G)/\rZ(T)\rO_p(G)$ is a chief factor of $G$, and then $\rZ(U) = \rZ(T)$ by Proposition~\ref{abecent}.  Thus we may assume $G \in \CEmbC(S)$.

Let $L = \rN_G(S)$.  Since $\rO^p(G) \cap S > 1$, we have $\rO^p(L) \cap S > 1$ by Theorem~\ref{yosh} and Theorem~\ref{tate}.  Applying Theorem~\ref{tate} again we see that $L'L^p \cap S \not= \Phi(S)$, which means that $L$ acts non-trivially on $S/\Phi(S)$.  At the same time, the action of $L$ on $S/\Phi(S)$ is reducible, since there is a proper non-trivial invariant subspace $\rO_p(G)\Phi(S)/\Phi(S)$: we have $\rO_p(G) < S$ since $S$ is not normal in $G$, so $\rO_p(G)\Phi(S) < S$ by the fact that $\Phi(S)$ is the intersection of all maximal closed subgroups of $S$, and we have $\rO_p(G) \not\le \Phi(S)$ by Corollary~\ref{tatecor}.  This establishes condition (a).

For condition (b), let $U = T\rO_p(G)/\rO_p(G) = \rF(G/\rO_p(G))$.  Note that $\rZ(U)$ is central in $G/\rO_p(G)$ by Proposition~\ref{abecent}, and $U/\rZ(U)$ is a chief factor of $G/\rO_p(G)$ since $G \in \CEmbC(S)$.  We are now in the situation of Lemma \ref{sympauto}, and so $p \cdot \ord(q,p)$ is even.

(iii) Let $R$ be the product of all $S$-conjugates of $Q$ and let $C = \rC_{SR}(R)$.  Then $SR/C$ is a $p'$-embedding of $SC/C$, so we may assume $G = SR$ and $\rC_G(R)=1$.  Moreover, $R$ is of the form $Q_1 \times \dots \times Q_n$ where $Q_i$ is an $S$-conjugate of $Q$.  Notice that $\rN_R(S)$ decomposes as
\[ \rN_{Q_1}(S_1) \times \dots \times \rN_{Q_n}(S_n),\]
where $S_i = S \cap Q_i$.  We have $\rN_{SR}(S) > S$ by Lemma~\ref{selftranslem}, so $\rN_{Q_i}(S) > S_i$ for some $i$; hence $\rN_Q(S) > S$.  Thus there is some element $x \in \rN_Q(S)$ of order $q$, where $q$ is a prime distinct from $p$.  Suppose that $S$ does not normalise $Q$; let $y \in S \setminus \rN_S(Q)$.  Then $x$ and $yxy^{-1}$ lie in distinct factors $Q_i$, so $z = xyx^{-1}y^{-1}$ has order $q$.  But $z$ is contained in $[S,\rN_Q(S)] \le S$, so $z$ is contained in a pro-$p$ group, a contradiction.\end{proof}

\section*{Acknowledgements} This paper is based on results obtained by the author as a PhD student at Queen Mary, University of London under the supervision of Robert Wilson, supported by EPSRC.  I would also like to thank Charles Leedham-Green for his advice and guidance during my doctoral studies.

\end{document}